\documentclass[12pt, reqno]{amsart}

\usepackage[a4paper,margin=1in]{geometry}
\usepackage{array}
\usepackage{amsmath,amsthm,amsfonts,amssymb,mathtools}
\usepackage{newtxtext,newtxmath} % modern replacement

\usepackage{mathrsfs}
\usepackage{enumerate}
\usepackage{xcolor}
\usepackage[all,cmtip]{xy}
\usepackage{tikz-cd} 
\usepackage{hyperref}
\usepackage[capitalise]{cleveref}
\usepackage{microtype}
\usepackage{graphicx}
\usepackage{enumitem}
\setlist[enumerate]{label=(\roman*)}

\hypersetup{
	colorlinks=true,
	linkcolor=black,
	citecolor=blue,
	urlcolor=blue,
	pdftitle={},
	pdfauthor={}
}

\newtheorem{thm}{Theorem}[section]
\newtheorem{lem}[thm]{Lemma}
\newtheorem{proposition}[thm]{Proposition}
\newtheorem{corol}[thm]{Corollary}
\newtheorem{ex}{ex}[section]

\newenvironment{lemma}{\begin{lem}\;}{\end{lem}}
\newenvironment{prop}{\begin{proposition}\;}{\end{proposition}}
\newenvironment{theorem}{\begin{thm}\;}{\end{thm}}
\newenvironment{definition}{\begin{defin}\;}{\end{defin}}
\newenvironment{cor}{\begin{corol}\;}{\end{corol}}

\theoremstyle{remark}
\newtheorem{remark}[thm]{Remark}
\newtheorem{example}[thm]{\bf Example}
\newtheorem{qstn}[ex]{\bf Question}
\newtheorem{defin}{\bf Definition}

\newcommand{\Hom}{{\rm Hom}}

\newcommand{\End}{{\rm End}}

\mathchardef\mhyphen="2D

\newcommand{\isomor}{\,\raisebox{4pt}{$\sim$}{\kern -.75em\to}\,}

\newcommand{\gm}{\mathbb G_m}

\newcommand{\ul}{\underline}

\newcommand{\hkr}{\hookrightarrow}

\newcommand{\inv}{{}^{-1}}

\newcommand{\topet}{\!{}_{\rm\acute{e}t}}

\def\PP{\mathbb P}
\def\OO{{\mathscr O}}

\newcommand{\pic}{{\rm Pic}}

\newcommand{\ZZ}{\mathbb Z}
\newcommand{\CC}{\mathbb C}

\title[Algebraic and Analytic Brauer groups of homogeneous spaces]{Algebraic and Analytic Brauer Groups of Homogeneous Spaces}
\author[S. Bhaumik and P. Saha]{Saurav Bhaumik \and Pinakinath Saha}
\address{\texttt Department of Mathematics, IIT Bombay, Mumbai 400076, India.}
\email{\texttt saurav@math.iitb.ac.in (S. Bhaumik)}
\address{\texttt Department of Mathematics, IIT Delhi, Delhi 110016, India.}
\email{\texttt pinakinath@maths.iitd.ac.in (P. Saha)}
\date{\today}

\begin{document}
	
	\begin{abstract}
		In this article, we compute both the algebraic and the analytic Brauer groups of a homogeneous space under the action of a connected, simply connected, semisimple complex algebraic group, where the stabilizer subgroup is closed, and connected.
	\end{abstract}

	\maketitle
	\tableofcontents

	\section{Introduction}
	Throughout this article, all varieties are defined over the field $\CC$ of complex numbers.
	Let $G$ be a connected, simply connected, semisimple complex algebraic group. Let $H$ be a closed connected subgroup of $G$. Consider the homogeneous space $G/H$, which is a smooth quasi-projective variety.
	\medskip
	
	Iversen \cite{Iversen} studied the Brauer group of a connected linear algebraic group defined over an algebraically closed field. Iversen showed that the Brauer group of a character free linear algebraic group is a finite group which is zero if and only if its fundamental group is cyclic. Haboush \cite{Haboush} studied the Brauer groups of homogeneous spaces. In \cite{Haboush}, Haboush constructed an injective group homomorphism  
	\begin{equation}\label{I:eq1}
		{\rm E}_{\rm al}(H,\gm)\longrightarrow {\rm Br}(G/H),
	\end{equation}
	where ${\rm E}_{\rm al}(H,\gm)$ is the set of equivalence classes of central $\gm$-extensions of $H$, which forms a group under Baer sum, and ${\rm Br}(G/H)$ is the Brauer group (classes of Azumaya algebras modulo an equivalence relations, cf. \cite[IV, \S 2]{Milne}) of $G/H$. The map in \cref{I:eq1} is 
	defined as follows. Given a class $\xi\in {\rm E}_{\rm al}(H, \gm)$, there is a projective representation $\rho_{\xi}:H\to {\rm PGL}_{n}$ such that the pullback of the central extension $1\to \gm \to {\rm GL_n}\to {\rm PGL}_n\to 1$ via $\rho_{\xi}$ is $\xi$. We then send  $\xi$ to the class $q_{*}(\rho_{\xi})$ of the $\PP^{n-1}$-fibration over $G/H$ associated to $\rho_{\xi}$ and the principal $H$-bundle $q: G\rightarrow G/H$ (cf.~\cite[Proposition~2.16 and Proposition~3.5]{Haboush}).
	
	 In this article, we prove that the map in \cref{I:eq1} is an isomorphism (cf. \cref{brauergp}).

     Every smooth algebraic variety $X$ gives rise to an analytic space in the sense of GAGA (cf. \cite{Serre}) which we still denote by $X$. Then there are natural maps \begin{equation}\label{pic-alg-an}
		\pic(X)\to \pic^{\rm an}(X),
	\end{equation}
	\begin{equation}\label{brauer-alg-an} 
		{\rm Br}(X)\to {\rm Br}^{\rm an}(X),
	\end{equation}
	where $\pic^{\rm an}(X)$ and ${\rm Br}^{\rm an}(X)$ denote respectively the analytic Picard group and analytic Brauer group of the analytic space $X$ (cf. \S\ref{prelim} for the precise definition).
	These maps are usually not isomorphisms (cf. \cref{ex-1} and \cref{ex-2}). 
	\medskip
	
	As before, let $G$ be a connected, simply connected, semisimple algebraic group, let $H$ be a closed connected subgroup, and let $M=G/H$. Then we prove the following results.
	\begin{enumerate}	
		\item The map in \cref{pic-alg-an} for $M$ induces an isomorphism at the level of $n$-cotorsions (\cref{brauer-3}).
	
		\item If $H$ is reductive, then the map in \cref{pic-alg-an} for $M$ is an isomorphism (cf. \cref{picard: alg-ana}). Moreover, we provide an example which shows that the map in \cref{pic-alg-an} is not necessarily an isomorphism when $H$ is not reductive (cf. \cref{remark: picard}).
		
		\item The map in \cref{brauer-alg-an} for $M$ is an isomorphism (cf. \cref{cor: baruer}). 
	\end{enumerate}
	\medskip
	
	Finally, we compute the Brauer group ${\rm Br}(M)$. One of our main result is that there is a natural isomorphism ${\rm Ext}^{1}(\pi_1(H),\ZZ)\simeq {\rm Br}(M)$ (cf.\cref{theorem5.5}), where $\pi_{1}(H)$ denotes the topological fundamental group of $H$.
	
	When $H$ is semisimple, we have an identification $X(\pi_{1}(H))\simeq {\rm E}_{\rm al}(H, \gm )$,  where $X(-):={\rm Hom}(-,\gm)$ (cf. \cite[Proposition 3.8.]{Iversen76}). Under this identification, the map in \cref{I:eq1} gives a map $X(\pi_1(H))\to {\rm Br}(M).$
	We show that this map can be re-interpreted in terms of highest weight theory (cf. \S 2, \cref{geomteric}).
	\medskip

	Let $P$ be a principal ${\rm PGL}_n$-bundle on a variety $X$. The obstruction to $P$ being Zariski locally trivial is given by a cohomology class in the Brauer group ${\rm Br}(X)$ (cf. \cite[IV]{Milne}). It is natural to ask if given any algebraic group $G$, the obstruction to a principal $G$-bundle on a variety $X$ being Zariski locally trivial can be found in the Brauer group ${\rm Br}(X)$. A consequence of our computation of the Brauer group of homogeneous varieties is that this is not possible. Indeed, we show that for any non-special simply connected semisimple algebraic group $H$ embedded in ${\rm SL}_N$, the homogeneous variety ${SL}_{N}/H$ has trivial Brauer group, while the principal $H$-bundle ${\rm SL}_{N}\to {\rm SL}_{\rm N}/H$ is not Zariski locally trivial.

	\section{Notations and Preliminaries}\label{prelim}
	
	In this section, we recall basic definitions and standard facts from the theory of algebraic and analytic Brauer groups. We refer the reader to \cite{Gro1}, \cite{Gro2}, \cite{Gro3}, \cite{Schroer05}, \cite{Iversen}, \cite{Artin-Mumford}, \cite{Beauvile}, \cite{Har-I}, and \cite{Milne}.
	
\medskip
	
{\em Notations.} Throughout this article, all varieties are defined over the field $\CC$ of complex numbers. For a variety $X$, we also denote by $X$ the associated analytic space. We denote the topological fundamental group of $X$ by $\pi_{1}(X)$. We denote ``$\simeq$" to indicate an isomorphism.
\medskip

\subsection{Azumaya algebras and principal ${\rm PGL}_r$-bundles}
From \cite{Gro1}, \cite[III, and IV]{Milne} we recall the following definition Azumaya algebra and principal ${\rm PGL}_r$-bundles and their correspondence.
\medskip

Let $X$ be a variety.
\begin{definition}
	An \emph{Azumaya algebra} $A$ over $X$ is a coherent sheaf of $\OO_{X}$-algebras such that for every $x\in X$, the stalk $A_x$ is locally free of finite rank as an $\OO_{X,x}$-module and the map 
	\[A_x \otimes_{\OO_{X,x}} A_x^{\circ} \to {\rm End}_{\OO_{X,x}}(A_x)\] sending $a\otimes a'$ to the endomorphism $(x\mapsto axa')$ is an isomorphism, where $A_x^{\circ}$ denotes the opposite algebra to $A_x$ (in other words, $A_x$ is an Azumaya algebra over the local ring $\OO_{X,x}$).
\end{definition}

\begin{definition}
Two Azumaya algebras $A$ and $A'$ over $X$ are
said to be \emph{Morita equivalent} if there exist locally free $\OO_X$-modules $E$ and $E'$ of finite ranks such that
\[
A \otimes_{\OO_X} \ul{\End}_{\OO_X}(E) \simeq
A' \otimes_{\OO_X} \ul{\End}_{\OO_X}(E').
\]
\end{definition}
Note that Morita equivalence is an equivalence relation, because for any two locally free $\OO_{X}$-modules $E$ and $E'$, we have $\ul{{\rm End}}(E)\otimes \ul{{\rm End}}(E') \simeq \ul{{\rm End}}(E\otimes E')$.

The set ${\rm Br}(X)$ of equivalence classes of
Azumaya algebras over $X$ forms a group under the operation $[A][A'] = [A \otimes A']$,
where the identity element is $[\OO_X]$ and $[A]^{-1}=[A^{\circ}]$. The group ${\rm Br}(X)$ is called the \emph{Brauer group} of $X$.
\medskip

There is a correspondence between isomorphism classes of Azumaya algebras of rank $r^2$ and of principal ${\rm PGL}_r$-bundles on $X$ locally trivial in the \'etale topology. Both are classified by the pointed set $H^1\topet(X, {\rm PGL}_r)$ (cf.~\cite[Theorem 2.5, IV]{Milne}). The correspondence is given as follows. Given an Azumaya algebra $A$ of rank $r^2$ over $X$, there exists an \'etale covering $\mathcal{U}=(U_i\to X)$ such that \[A\otimes_{\OO_{X}}\OO_{U_i}\simeq {\rm M}_{r}(\OO_{U_i})\] for each $i$. Since $\ul{{\rm Aut}}(M_{r}(\OO_X))={\rm PGL}_r$, $A$ determines a principal ${\rm PGL}_r$-bundle on $X$, thus a class in $H^{1}_{\topet}(X, {\rm PGL}_r)$. Conversely, given a ${\rm PGL}_r$-bundle $P$ on $X$, the associated bundle $P\times^{{\rm PGL}_r} {\rm M}_r$ is an Azumaya algebra over $X$ of rank $r^2$. These two constructions are inverse to each other up to isomorphism.
\medskip

Let $P_1\to X$ be a principal ${\rm PGL}_{r_1}$-bundle and $P_2\to X$ be a principal ${\rm PGL}_{r_2}$-bundle in the \'etale topology. The homomorphism
$$
{\rm PGL}_{r_1}\times {\rm PGL}_{r_2}\longrightarrow {\rm PGL}_{r_1r_2},\qquad (\alpha, \alpha')\mapsto \alpha\otimes \alpha'
$$
yields a principal ${\rm PGL}_{r_1r_2}$-bundle $P_1\otimes P_2$. Two bundles $P$ and $P'$ are called \emph{Brauer equivalent} if there exist locally free $\OO_X$-modules $E_1,E_2$ of ranks $r_2,r_1>0$ respectively, such that
$P\otimes \PP(E_1)\simeq P'\otimes \PP(E_2)$.
\medskip

Let $A$ be an Azumaya algebra of rank $r^2$, whose corresponding ${\rm PGL}_r$-bundle is $P$. Then the Azumaya algebra corresponding to $P\otimes \PP(E)$ is isomorphic to $A\otimes_{\OO_X} \ul{{\rm End}}_{\OO_X}(E)$. It follows that two bundles are Brauer equivalent if and only if their Azumaya algebras are Morita equivalent. 

Let $A$ and $A'$ be the Azumaya algebras whose corresponding bundles are $P$ and $P'$, respectively. Then the Azumaya algebra $A\otimes_{\OO_{X}} A'$ corresponds to the bundle $P\otimes P'$.

Therefore the Brauer group ${\rm Br}(X)$ can be thought as the group of Brauer equivalence classes of principal ${\rm PGL}_r$-bundles over $X$ for $r\ge 1$, where the group law is induced by the tensor product of principal bundles.

\subsection{Algebraic and Analytic Cohomological  Brauer Groups}

	As $\ul{{\rm Aut}}(\PP^{r-1})={\rm PGL}_r$,	the $\PP^{r-1}$-bundles over $X$, locally trivial in the \'etale topology are classified by the cohomology set $H^{1}_{\topet}(X,{\rm PGL}_{r})$.
	The short exact sequence of sheaves
	\[
	1 \to \gm \to {\rm GL}_{r} \to {\rm PGL}_{r} \to 1
	\]
	gives a connecting map
	\[
	H^{1}_{\topet}(X,{\rm PGL}_{r}) \longrightarrow H^{2}_{\topet}(X,\gm).
	\]
	This associates to each $\PP^{r-1}$-bundle a torsion class in  $H^{2}_{\topet}(X,\gm)$. The class measures the obstruction to coming from an actual vector bundle.  It is zero exactly when the $\PP^{r-1}$-bundle is the projectivization of some rank $r$-vector bundle.

	\begin{definition}
		The \emph{cohomological Brauer group} of $X$ is defined as the torsion subgroup of $H^2_{\topet}(X,\gm)$ and it is denoted by ${\rm Br}'(X)$.
	\end{definition}

	\begin{definition}
		The \emph{analytic Brauer group} of $X$ is defined in analogy with the algebraic case and it is denoted by ${\rm Br}^{\rm an}(X)$. 
	\end{definition}
	
	\begin{definition}
		The \emph{analytic cohomological Brauer group} of $X$ is defined as the torsion subgroup of $H^2(X, \OO_{X}^{an, \times})$, where $\OO_{X}^{an}$ denotes the sheaf of holomorphic functions on $X$, viewed as an analytic space. It is denoted by ${\rm Br'}^{\rm an}(X)$. 
	\end{definition}

	There are inclusions,
	 ${\rm Br}(X)\hookrightarrow {\rm Br}'(X)$ and
	$
	{\rm Br}^{\rm an}(X)\hookrightarrow {\rm Br'}^{\rm an}(X),
	$
	which are not surjective in general. Grothendieck~\cite{Grothendieck} asked whether these are bijective. This is a major open problem in the theory of Brauer groups.
	Grothendieck himself showed that equality holds for smooth
	algebraic surfaces \cite{Grothendieck}, and  Schr\"oer treated the case of algebraic surfaces with isolated singularities \cite{Schroer}.
	
	By Gabber's theorem the map ${\rm Br}(X)\to {\rm Br}'(X)$ is an isomorphism for quasi-projective varieties (cf. \cite[Theorem~1.1.]{Jong2005ARO}). If $X$ is a smooth variety, then ${\rm Br'}(X)=H^2\topet(X,\gm)$.

	For a variety $X$,   we denote the algebraic (respectively, analytic) Picard group of $X$ by ${\rm Pic}^{}(X)$ (respectively, ${\rm Pic}^{\rm an}(X)$).
	There is a natural map ${\rm Pic}^{}(X)\to {\rm Pic}^{\rm an}(X),$ which takes an algebraic line bundle to the corresponding analytic line bundle, is not a bijection in general. If $X$ is projective, then this map is an isomorphism (cf. \cite{Serre}).	
	
A complex-analytic space $X$ is called \emph{Stein} if $H^{j}(X, \mathcal{F}) = 0$ for all $j\ge 1$ and all coherent $\OO_X^{an}$-modules $\mathcal{F}$. If $X$ is an algebraic affine variety, then the associated complex analytic space is Stein (cf. \cite{Har-II}).

	\begin{example}\label{ex-1}
		Let $X=\gm\times \gm$ over $\CC$. Since $X$ is smooth affine and its coordinate ring is a UFD, the algebraic Picard group ${\rm Pic}^{}(X)$ is trivial. On the other hand, $X$ is Stein; by using the exponential exact sequence
		$$
		0\to\ul{\ZZ} \to \OO_{X}^{an} \to \OO^{ an, \times}_{X}\to 1,
		$$
		we have 
		$$
		{\rm Pic}^{\rm an}(X)=H^2(X,\ul{\ZZ})=H^2_{\rm sing}(S^1\times S^1,\ZZ)=\ZZ.
		$$
	\end{example}
	
	\medskip
	
	For a variety $X$, there is a natural surjective map
	$$
	{\rm Br'}^{}(X)\longrightarrow {\rm Br'}^{\rm an}(X)
	$$
	(cf. \cite[Proposition~1.3.]{Schroer05}), which is not injective in general. If $X$ is projective, this map is an isomorphism (cf. \cite[Proposition~1.3.]{Schroer05}).
	
	\begin{example}\label{ex-2}
		Let $X=\gm\times \gm$ over $\CC$. Since $X$ is smooth, we have ${\rm Br'}^{}(X)=H^2_{\topet}(X,\gm)$. From the Kummer sequence
		$$
		1\to \mu_n\to \gm \xrightarrow{(\cdot)^n} \gm \to 1,
		$$
		we obtain an exact sequence
		$$
		H^1_{\topet}(X,\gm)\to H^2_{\topet}(X,\mu_n)\to H^2_{\topet}(X,\gm)\xrightarrow{n} H^2_{\topet}(X,\gm).
		$$
		Since ${\rm Pic}(\mathbb{A}^2)=0$, and $X$ is an open subset of $\mathbb{A}^2$, we have ${\rm Pic}(X)=H^1_{\topet}(X,\gm)=0$. 
		Hence
		\[
		H^2_{\topet}(X,\mu_n)=\ker\big(H^2_{\topet}(X,\gm)\xrightarrow{n} H^2_{\topet}(X,\gm)\big).
		\]
		By \cite[Theorem 3.12]{Milne},
		\[ H^2_{\topet}(X,\mu_n)\simeq H^2_{\rm sing}(X(\CC),\ZZ/n\ZZ)\simeq H^2_{\rm sing}(S^1\times S^1,\ZZ/n\ZZ)\simeq \ZZ/n\ZZ.\]
		Thus $H^2_{\topet}(X,\gm)\neq 0$. On the other hand, $X$ is Stein; by using the exponential exact sequence
		$
		0\to\ul{\ZZ} \to \OO_{X}^{an} \to \OO_{X}^{an, \times}\to 1
		$, and using $H^3_{\rm sing}(S^1\times S^1,\ZZ/n\ZZ)=0$,
	it is easy to see that ${\rm Br'}^{\rm an}(X)=0$.
	\end{example}

\subsection{Azumaya representations and Brauer groups}\label{Sec3}
In this section, we shall review certain results on the Brauer groups of the homogeneous spaces from the Haboush's paper (cf. \cite{Haboush}).
	
Throughout this section, $H$ denotes a complex algebraic group.
\begin{definition}
An \emph{Azumaya representation} of $H$ of \emph{degree} $n$ is a morphism of algebraic groups $$\rho: H\to {\rm PGL}_{n}.$$
\end{definition}

An Azumaya representation $\rho$ is called \emph{trivial} if there exists a lift $\widetilde{\rho}:H\to {\rm GL}_{n}$ such that the following diagram commutes:
\[
\xymatrix{& &{\rm GL}_{n} \ar[d]^{\pi}\\
			H\ar@{-->}[urr]^{\widetilde{\rho}}\ar[rr]_{\rho}
			& & {\rm PGL}_{n}}
\]
where $\pi$ is the natural quotient map.

Two Azumaya representations $\rho: H\to {\rm PGL}_{n}$ and $\rho': H\to {\rm PGL}_{n'}$ are called \emph{equivalent} if there exist trivial representations
$\alpha: H\to {\rm PGL}_{r}$ and $\alpha': H\to {\rm PGL}_{r'}$ such that the representations $\rho\otimes \alpha$ and $\rho'\otimes \alpha'$ are isomorphic.
	
The set ${\rm RBr}(H)$ of equivalence classes of Azumaya representations forms a monoid under tensor product.
	
Recall that ${\rm M}_{n}$ denotes the algebra of $n\times n$ matrices. By the Skolem--Noether theorem, ${\rm Aut}({\rm M}_{n})\simeq {\rm PGL}_n$, with ${\rm PGL}_n$ acting on ${\rm M}_{n}$ by the conjugation action. 
Thus, if $\rho:H \to {\rm PGL}_n$ is an Azumaya representation, we may regard it as a map $\rho: H\to {\rm Aut}({\rm M}_n)$. Let ${\rm M}_{n}^{\rm op}$ be the opposite algebra. The transpose map $A\mapsto {}^{t}\! A$ induces an isomorphism $\tau: {\rm M}^{\rm op}_{n}\to {\rm M}_n$. Transporting the action via $\tau$ defines the \emph{opposite representation} $\rho^0$. 
For $h\in H$, $\rho(h)(A) = uAu^{-1}$ for all $A\in {\rm M}_n$, then
	\[
	\rho^{0}(h)(A) ={}^{t}\!{u^{-1}} A~^{t}\!u,
	\]
where $^{t}\!u$ denotes the transpose of $u$.
	It follows that $\rho\otimes \rho^{0}$ is trivial for every projective representation $\rho$. Hence, ${\rm RBr}(H)$ is in fact a group, called the \emph{representation theoretic Brauer group} of $H$ (cf.~\cite[Lemma~2.8]{Haboush}).
	
	\begin{remark}
		The Brauer group ${\rm Br}(H)$ and the representation theoretic Brauer group ${\rm RBr}(H)$ are not the same objects. Representation theoretic Brauer group of $H$ carries the information of algebraic central extension of $H$ by $\gm$. For instance, if $H={\rm SO}_n$, then ${\rm Br}(H)=0$ (cf. \cite[Corollary 4.3.]{Iversen}), but ${\rm RBr}(H)=\ZZ/2\ZZ$.
	\end{remark}
	
	\subsection{Algebraic central extensions}
	In this subsection, we recall certain basic definitions.
	
	Let $A$ be a diagonalizable complex algebraic group.
	\begin{definition}
		An \emph{algebraic central extension} of $H$ by $A$ is a short exact sequence of algebraic groups
		\[
		E \, : 1\to A \xrightarrow{j} E \xrightarrow{p}  H \to 1
	    \]
		such that $j(A)$ lies in the center of $E$.
	\end{definition}
	
	\begin{definition}
		An algebraic central extension of $H$ by $A$ 
		\[
		E\, : 1\to A \xrightarrow{j} E \xrightarrow{p}  H \to 1
		\]
		is called \emph{split} if there exists a homomorphism $\sigma: H\to E$ with $p\circ \sigma ={\rm id}_{H}$.
	\end{definition}
	
	If $E_1$ and $E_2$ are two extensions of $H$ by $A$, then by a homomorphism $f: E_1\to E_2$ means the following commutative diagram:
	$$
	\xymatrix{
		1\ar[r] & A  \ar[d]_-{{\rm id }} \ar[r]& E_1\ar[r]\ar[d]^{f} & H\ar[r]\ar[d]^{\rm id}& 1\\
		1\ar[r]&A\ar[r]& E_2 \ar[r] &H\ar[r]& 1}
	$$
	
	Two extensions are equivalent if such an $f$ exists. By the short five lemma $f$ is necessarily an isomorphism.
	\medskip
	
	The Baer sum of two extensions
	$$
E_1: \, \, 1\to A \xrightarrow{f_1} E \xrightarrow{g_1}  H \to 1 ,\qquad
E_2: \, \, 1\to A \xrightarrow{f_2} E \xrightarrow{g_2}  H \to 1,
	$$
	is constructed as follows. Form the fiber product
	$$
	E_{1}\times_{H} E_{2}=\{(e_1,e_2)\in E_1\times E_2\mid g_1(e_1)=g_2(e_2)\},
	$$
	and quotient by the subgroup $\{(f_1(x),{f_{2}(x)}^{-1})\mid x \in A \}$. The resulting group \[ E_{1}\times_{H} E_{2} \Big\slash \{(f_1(x),{f_{2}(x)}^{-1})\mid x \in A \}\] is again a central extension of $H$ by $A$, denoted $E_1+E_2$. Thus, the sum $E_1 + E_2$ defines the extension
	\[
	1 \;\longrightarrow\; A \xlongrightarrow{f_1} E_1 + E_2 \xlongrightarrow{g_1} H \;\longrightarrow\; 1,
	\]
	where the inclusion map is
	\[
	x \;\longmapsto\; \bigl[(f_1(x), \mathrm{id})\bigr] \;=\; \bigl[(\mathrm{id}, f_2(x))\bigr],
	\]
	and the projection map is
	\[
	(e_1,e_2) \;\longmapsto\; g_{1}(e_1) \;=\; g_{2}(e_2).
	\]
	
	We give an equivalent description of $E_1 + E_2$ is defined as follows.  
	First, consider the product extension
	\[
	E_1 \times E_2 : \quad 1 \;\longrightarrow\; A \times A \;\longrightarrow\; E_1 \times E_2 \;\longrightarrow\; H \times H \;\longrightarrow\; 1,
	\]
	and pull it back along the diagonal map $\Delta_H : H \to H \times H$.  
	This yields the central extension
	\[
	\Delta_H^{*}(E_1 \times E_2) : \quad 1 \;\longrightarrow\; A \times A \;\longrightarrow\; E_1 \times_H E_2 \;\longrightarrow\; H \;\longrightarrow\; 1.
	\]
	Next, we form the push-out of this extension along the addition map $\nabla_A : A \times A \to A$.  
	The resulting central extension of $H$ by $A$ is precisely the Baer sum $E_1 + E_2$.
	Thus, the bottom sequence is a central extension of $H$ by $A$, namely the Baer sum $E_1 + E_2$:  
	\begin{equation*}\label{Baer sum}
		\begin{array}{c@{\hspace{2cm}}c}
			\xymatrix{
				1 \ar[r] 
				& A \times A \ar@{=}[d] \ar[r] 
				& E_1 \times E_2 \ar[r] 
				& H \times H \ar[r] 
				& 1 \\
				1 \ar[r] 
				& A \times A \ar[d]_{\nabla_A} \ar[r] 
				& E_1 \times_H E_2 \ar[r] \ar[d] \ar[u] 
				& H \ar[r] \ar@{=}[d] \ar[u]_{\Delta_H} 
				& 1 \\
				1 \ar[r] 
				& A \ar[r] 
				& E_1 + E_2 \ar[r] 
				& H \ar[r] 
				& 1
			}
		\end{array}
	\end{equation*}
	The set of equivalence classes of central extensions of $H$ by $A$ is denoted by ${\rm E}_{\rm al}(H, A)$. This set forms an abelian group under the Baer sum: the sum of two equivalence classes of extensions is defined to be the equivalence class of their Baer sum.  
	The zero element in ${\rm E}_{\rm al}(H, A)$ is the class of a split central extension.

	\subsection{The Haboush isomorphism}
   Assume that $H$ is a closed subgroup of a connected complex algebraic group $G$.
    
	In \cite{Haboush}, Haboush showed that there is a natural isomorphism
	\begin{equation}\label{haboush}
		\varepsilon :{\rm RBr}(H)\xrightarrow{\;\simeq\;} {\rm E}_{\rm al}(H, \gm),
	\end{equation}
	 which send the class of $\rho: H\to {\rm PGL}_n$ to the class $E_{\rho}$ of the pull back of the central extension \[1\to \gm \to {\rm GL}_n\to {\rm PGL}_n\to 1\] via $\rho$ (cf.~\cite[Proposition~2.16]{Haboush}).
	
	 Consider the principal $H$-bundle
	$$
	q: G\to G/H.
	$$
	Haboush defined an induction morphism
	\begin{equation}\label{induction}
		i_{G/H}: {\rm RBr}(H)\to {\rm Br}(G/H),
	\end{equation}
	sending a class $\xi\in {\rm RBr}(H)$, represented by $\rho:H\to {\rm PGL}_n$, to the class $q_{*}(\rho)$ of $\PP^{n-1}$-fibration over $G/H$, associated to $\rho$, and $q$ (cf.~\cite[Proposition~3.5]{Haboush}).  
	
	Haboush \cite[Theorem 5.1]{Haboush} proved that when $G$ is semisimple, simply connected, and $H\subset G$ is a closed connected subgroup, the map ${i}_{G/H}$ in \eqref{induction} is injective.
	Therefore combining \eqref{haboush} and \eqref{induction} we obtain, for $G$ semisimple simply connected and $H\subset G$ closed and connected, a canonical injective map
\begin{equation}\label{Haboush: injective}
{\rm E}_{\rm al}(H, \gm)\hookrightarrow {\rm Br}(G/H).	
\end{equation}

	In this article, we prove that \eqref{Haboush: injective} is surjective as well.
	
\begin{theorem}\label{brauergp}
	Let $G$ be a connected, simply connected, complex semisimple algebraic group and
	$H\subset G$ be a closed connected subgroup.
	Then the natural homomorphism
	$$
	{\rm E}_{\rm al}(H, \gm) \to {\rm Br}(G/H)
	$$
	is bijective.
\end{theorem}
We prove Theorem \ref{brauergp} in Section \S \ref{Sec7}.

\subsection{The group ${\rm E}_{\rm al}(H,\gm)$ and $X(\pi_1(H))$}

	Let $H$ be semisimple. Let $\pi_{1}(H)$ be the fundamental group of $H$. Note that since $H$ is semisimple, the topological and algebraic fundamental group of $H$ are same. 
	Denote by $X(\pi_{1}(H))={\rm Hom}_{\rm gp}(\pi_{1}(H), \gm)$, the group of characters of $\pi_{1}(H)$. There is a natural map
	\begin{equation}\label{charcter center}
	X(\pi_{1}(H))\to {\rm E}_{\rm al}(H, \gm)
	\end{equation}
	sending $\chi$ to the class of the pushout $\chi_{*}$ of the universal covering sequence
	$$
	1\to \pi_{1}(H)\to H_{sc}\to H\to 1
	$$
	via $\chi$. This map \eqref{charcter center} is an isomorphism (cf.~\cite[Proposition~3.8]{Iversen76}). 
	Note that when $H$ is not semisimple, there is no such isomorphism as in \eqref{charcter center}. For example, for $H=\gm$, the group ${\rm E}_{\rm al}(H, \gm)$ is trivial, but the group $X(\pi_{1}(H))$ is not.
	
	When $G$ is a semisimple, simply connected, algebraic group over $\CC$ and $H$ is a closed connected semisimple subgroup of $G$, Haboush showed (using \eqref{haboush}, \eqref{induction}, and \eqref{charcter center}) that the natural homomorphism
	\begin{equation}\label{character-brauer}
		X(\pi_{1}(H)) 
		\;\xrightarrow{\;\simeq\;}\; {\rm E}_{\rm al}(H,\mathbb{G}_m) 
		\;\xrightarrow{\;\simeq\;}\; {\rm RBr}(H) 
		\;\hookrightarrow\; {\rm Br}(G/H)
	\end{equation}
	is injective.

    \subsection{Projective representations and geometric constructions}\label{projective}
     Let $H$ be a semisimple. Let $\pi: H_{sc}\to H$ be the universal cover of $H$, with kernel $\pi_1(H)$, a finite central subgroup. Let $T_{sc}$ be a maximal torus of $H_{sc}$, and set $T=T_{sc}/\pi_1(H)$, the corresponding maximal torus of $H$.
    
    Let $Z\subset H_{sc}$ be a central finite subgroup, and let $H'=H_{sc}/Z$. If $\rho: H_{sc}\to {\rm GL}(V)$ is a linear representation of $H_{sc}$ on which $Z$ acts by a character $\chi$, then $\rho$ gives rise to a projective representation $\rho(V):H'\to {\rm PGL}(V)$. 
    
    \begin{lemma} \label{equiv-8} If $V_1,$ and $V_2$ are two linear representations of $H_{sc}$ such that $Z$ acts on both $V_1$ and $V_2$ by the same character $\chi$, then $\rho(V_1), \rho(V_2)$ are equivalent in ${\rm RBr}(H')$. \end{lemma}
    
    \begin{proof}
     In view of the above, it is enough to show that there are linear representations $W_1,W_2$ of $H_{sc}$ such that $Z$ acts trivially on both $W_1,W_2$ (i.e., $\rho(W_i)$ lifts to ${\rm GL}(W_i)$ as representations of $H'$) such that $V_1\otimes W_1\simeq V_2\otimes W_2$ as representations of $H_{sc}$. Since $Z$ is finite, there is $n$ such that $\chi^n$ is trivial. Choose $n_1,n_2$ two positive integers such that $n$ divides $n_1+n_2$. Consider $W_1=V_1^{\otimes n_2}\otimes V_2^{\otimes n_1}$ and $W_2=V_1^{\otimes n_2+1}\otimes V_2^{\otimes n_1-1}$. Then on both $W_1,W_2$, the subgroup $Z$ acts trivially. Moreover, as representations of $H_{sc}$, $V_1\otimes W_1\simeq V_2\otimes W_2$. So, $\rho(V_1\otimes W_1)$ is isomorphic to $\rho(V_2\otimes W_2)$ as representations of $H'$. Therefore, their classes in ${\rm RBr}(H')$ are equal. 
    \end{proof}

    \begin{lemma}\label{lemma : representation}
    	The short exact sequence $1\to \pi_{1}(H)\to T_{sc}\to T\to 1$ induces a short exact sequence of abelian groups
    	$$
    	1\to  X(T) \to X(T_{sc})\to X(\pi_{1}(H))\to 1.
    	$$
    \end{lemma}
    \begin{proof}
    	By \cite{Iversen}, the exact sequence induces
    	$$
    	1\to X(T)\to X(T_{sc})\to X(\pi_1(H))\to {\rm Pic}(T)\to \cdots.
    	$$
    	Since $T$ is smooth and its affine coordinate ring is a UFD, ${\rm Pic}(T)=0$.
    	The proof of the lemma follows.
    \end{proof}
    
    Let $\lambda$ be an integral dominant weight of $H_{sc}$ and let $V_\lambda$ be the corresponding irreducible representation. By Schur's lemma, $\pi_1(H)$ acts on $V_\lambda$ via a character.
    
     \begin{lemma}\label{equiv-9} If $\lambda_1, \lambda_2$ are dominant weights of $T_{sc}$, then $\rho(V_{\lambda_1})\otimes \rho(V_{\lambda_2})$ and $\rho(V_{\lambda_1+\lambda_2})$ are equivalent in ${\rm RBr}(H)$.
     \end{lemma}
    
    \begin{proof}
    Since $V_{\lambda_i}$, $i=1,2$ are highest weight representations of $H_{sc}$, $\pi_1(H)$ acts on $V_{\lambda_1}\otimes V_{\lambda_2}$ by the restriction of $\lambda_1+\lambda_2$ to $\pi_{1}(H)$. Also, $\pi_1(H)$ acts on $V_{\lambda_1+\lambda_2}$ by $\lambda_1+\lambda_2$ . The rest of the proof of the lemma follows from Lemma \ref{equiv-8}.
    \end{proof}
    
    Let $X(T_{sc})_{\ge0}$ be the monoid consisting of the dominant characters of $(H_{sc}, T_{sc})$. Then by \cref{equiv-9}, the association \[X(T_{sc})_{\ge 0}\rightarrow {\rm RBr}(H):\lambda\mapsto [\rho(V_{\lambda})]\] defines a homomorphism of monoids. 
    Since $X(T_{sc})$ is the Grothendieck completion of $X(T_{sc})_{\ge 0}$, this extends to a homomorphism of groups \[X(T_{sc})\to {\rm RBr}(H).\] Further, if $\lambda$ is a character of $T$, then its restriction to $\pi_1(H)$ is trivial, hence $\rho(V_\lambda)$ lifts to ${\rm GL}(V_\lambda)$. This means, by \cref{lemma : representation} the homomorphism \[X(T_{sc})\to {\rm RBr}(H)\] factors through
    \begin{equation}\label{c-prime} 
    c':X(\pi_1(H))\to {\rm RBr}(H).
    \end{equation}
    
    \medskip
    
    For a character $\chi\in X(\pi_1(H)$, consider the push-forward of the central extension \[1\to \pi_1(H)\to H_{sc}\to H\to 1\] by $\chi:\pi_1(H)\to\gm$, which gives a central extension of $H$ by $\gm$. By the work of Haboush, this corresponds to a unique element $c(\chi)\in {\rm RBr}(H)$. This gives rise to a map
  
    \begin{equation}\label{c} 
    	c:X(\pi_1(H))\to {\rm RBr}(H).
    \end{equation}

    \begin{lemma}\label{equiv-10} If $\rho:H_{sc}\to {\rm GL}(V)$ is a linear representation whose restriction to $\pi_1(H)$ acts by the character $\chi$, then all squares in the following diagram commute, where $K=(\gm\times H_{sc})/\pi_1(H)$ and $\rho_1([t,g])=\chi(t)\rho(g)$.
    	
    \centerline{\xymatrix{
    			1 \ar[r] & \pi_1(H)\ar[d]_\chi \ar[r] & H_{sc}\ar[d]\ar[r] & H\ar@{=}[d]\ar[r]& 1\\
    			1\ar[r] & \gm\ar@{=}[d]\ar[r] & K\ar[d]^{\rho_1}\ar[r] & H\ar[d]^{\rho(V)}\ar[r] & 1\\
    			1\ar[r] &\gm\ar[r] & {\rm GL}(V)\ar[r] & {\rm PGL}(V)\ar[r] &1	}}
    \end{lemma}
    \begin{proof}
    The proof follows immediately.
    \end{proof}
 
    \begin{prop} 
    	The maps $c'$ and $c$ defined above coincide.
    \end{prop}
    \begin{proof}
    In view of Haboush's work, it is enough to show that the central extensions of $H$ by $\gm$ corresponding to $c'(\chi)$ and $c(\chi)$ are the same for every $\chi\in X(\pi_1(H))$. In fact, it is enough to check this for those $\chi$ which come from dominant weights of $T_{sc}$. So, let $\chi$ be given by a dominant weight of $T_{sc}$. Note that the central extension of $H$ by $\gm$ corresponding to $c(\chi)\in {\rm RBr}(H)$ corresponds to the push forward of $1\to \pi_1(H)\to H_{sc}\to H\to 1$  by $\chi$. On the other hand, $c'(\chi)$ is constructed using a representation $\rho:H_{sc}\to {\rm GL}(V_{\chi})$ whose restriction to $\pi_1(H)$ is the character $\chi$, and $c'(\chi)$ is given by the corresponding projective representation $\rho(V_{\chi}):H\to {\rm PGL}(V)$.
    So, the central extension of $H$ by $\gm$ corresponding to $c'(\chi)$ is given by the pullback of $1\to \gm\to {\rm GL}(V_{\chi})\to {\rm PGL}(V_{\chi})\to 1$ by $\rho(V_{\chi}):H\to {\rm PGL}(V_{\chi})$. However, by \cref{equiv-10}, both these are the same as the second row in the diagram of \cref{equiv-10}.
    \end{proof}

   \subsection{Geometric interpretation of the map in  \eqref{character-brauer}}\label{geomteric}
   Let $G$ be a connected simply connected semisimple algebraic group and let $H$ be a closed connected subgroup. Suppose $H$ is also semisimple.
   Let $\lambda$ be an integral dominant weight of $H_{sc}$ and let $V_\lambda$ be the corresponding irreducible representation. By Schur's lemma, $\pi_1(H)$ acts on $V_\lambda$ via a character, so $H$ acts naturally on $\PP(V_\lambda)$. This give an Azumaya representation $\rho_\lambda: H\to {\rm PGL}(V_\lambda)$, which defines an $H$-action on $G\times {\rm End}(V_\lambda)$ by
   $$
   (g',A)\cdot h=(g'h^{-1},h\cdot A).
   $$
   The quotient $G\times^{H}{\rm End}(V_\lambda)$ is an Azumaya algebra on $G/H$, denoted by $q_{*}(\rho_\lambda)$, where $q:G\to G/H$.
    Consider
    $$
    P'_\lambda=G \times \PP(V_{\lambda}),
    $$
    where $H$ acts diagonally by
    $$
    h\cdot(g',p)=(g'h^{-1},h\cdot p).
    $$
    Then the quotient $P_\lambda=P'_\lambda/H$ is a $\PP(V_\lambda)$-fibration over $G/H$.

     By our earlier discussion both $P_{\lambda}$ and $q_{*}(\rho_\lambda)$ represent the same class in ${\rm Br}(G/H)$ (cf. \cref{projective}).
    
    Thus we obtain a natural homomorphism
    $$
    c: X(T_{sc})\to {\rm Br}(G/H), \qquad \lambda\mapsto P_{\lambda}.
    $$
    This map is trivial on $X(T)$; hence by Lemma~\ref{lemma : representation} it induces
    $$
    \tilde{c}: X(\pi_1(H))\to {\rm Br}(G/H),
    $$
    which is injective. This map is the composition of the maps given in \eqref{induction} and \eqref{c-prime} and it also the same map as in \eqref{character-brauer}.

    \begin{example} Consider $H= {\rm SO}_2$ and $G={\rm SL}_2$.  Note that $H$ is isomorphic to $\gm$ and ${\rm E}_{\rm al}(H,\gm)=0$. By Theorem~\ref{brauergp}, ${\rm Br}(X)=0$. Therefore, we conclude that any $\PP^n$-fibration over $X={\rm SL}_2/H$ is Zariski locally trivial. 
    \end{example}
    
    \begin{example}\label{example: brauer group non-zero}
    	For $n\ge 3$, consider $H={\rm SO}_n$ and $G={\rm SL}_n$. Consider the highest weight representation $V_{\lambda}$ of $H_{sc}={\rm Spin}_n(\CC)$ with highest weight $\lambda=\varpi$, where $\varpi$ denotes the fundament weight corresponding to the end node of the associated Dynkin diagram of $H_{sc}$. 
    	Note that $\pi_{1}(H)=\ZZ/2\ZZ$, and the restriction of $\lambda$ on $\pi_{1}(H)$ is not trivial. By Haboush's theorem we have
    	$$
    	X(\pi_1(H))\hookrightarrow {\rm Br}(X).
    	$$
    	Consider the action of $H$ on $\PP(V_\lambda)$. Then, the associated $\PP^{N-1}$-fibration
    	$$
    	P_{\lambda}:= G\times_{H} \PP(V_\lambda)\to G/H,
    	$$
    	with $N=\dim V_\lambda$, does not arise from a vector bundle and hence is not Zariski locally trivial.
    	In fact, by Theorem~\ref{brauergp}, it follows that $P_{\lambda}$ is a unique (up to Brauer equivalence) non trivial projective bundle  which does not arise from a vector bundle. 
    \end{example}

	\subsection{Picard groups of reductive groups}\label{Sec4}
	In this subsection, we prove that the Picard group of any connected reductive group is equal to the Picard group of its derived subgroup.
		
	Let $H$ be a connected reductive group. Let $H^{(1)}$ or $(H, H)$ denote the derived subgroup of $H$. Note that $H^{(1)}$ is a connected semisimple group. Let $Z_H$ denote the neutral component of the center of $H$. Then $H=H^{(1)}\cdot Z_H$ and $H^{(1)}\cap Z_{H}$ is finite (cf. \cite{Borel}). Moreover $$H/H^{(1)}\simeq Z_{H}\big/ (Z_{H}\cap H^{(1)}).$$ 
	If $T$ is a maximal torus of $H$, then $Z_H\subset T$. Consider the canonical projection $q: H\to H/H^{(1)}$. Since $H$ is connected reductive, $H/H^{(1)}$ is a torus. Thus we have a short exact sequence 
	\begin{equation}\label{question: split}
	1\to H^{(1)}\to H\xrightarrow{q} H/H^{(1)}\to 1 .
	\end{equation}
	\begin{qstn}
	A natural and fundamental question is whether short exact sequence \cref{question: split} splits, in other words whether
	there exists an algebraic group homomorphism $\sigma: H/H^{(1)}\to H$ with $q \circ \sigma=\mathrm{id}_{H/H^{(1)}}$. 
	\end{qstn} 
	
	We provide an affirmative answer to this question. To proceed we prove the following lemma.	
	\begin{lemma}\label{lemma: split}
		Let $T_{1},T_2,T_3$ be tori. If
		$$
		1\to T_1\xrightarrow{f} T_2\xrightarrow{g} T_3\to 1
		$$
		is a short exact sequence, then $f$ (and hence $g$) admits a splitting.
	\end{lemma}
	\begin{proof}
		The exact sequence induces a short exact sequence of character groups
		$$
		0\to X(T_3)\xrightarrow{g^*} X(T_2)\xrightarrow{f^*} X(T_1)\to 0.
		$$
	Since $T_{i}$ is torus, $X(T_i)$ is free abelian group. Since each $X(T_1)$ is a free abelian group, the sequence splits, and dualizing yields a splitting of the original sequence.
	\end{proof}
	
	Let $T_{H}$ be a maximal torus of $H$. Note that $q(T_{H})=H/H^{(1)}$. Set $T=H^{(1)}\cap T_{H}$. Then $T$ is a maximal torus of $H^{(1)}$. We obtain a short exact sequence of tori
	$$
	1\to T\to T_{H}\to q(T_{H})\to 1,
	$$
	which splits by Lemma~\ref{lemma: split}. Let $\sigma : q(T_H)\to T_{H}$ be a section of $T_{H}\to q(T_H)$. Then $\sigma(q(T_{H}))$ is (non-canonically) isomorphic to $Z_{H}$.
	
	Thus we have a commutative diagram
	$$
	\xymatrix{
		1 \ar[r] & T \ar@{^{(}->}[d] \ar[r] & T_H \ar@{^{(}->}[d] \ar[r] & q(T_H) \ar@/^1pc/[l]^{\sigma}  \ar@{=}[d] \ar[r] & 1 \\
		1 \ar[r] & H^{(1)} \ar[r] &  H\ar[r] & H/H^{(1)} \ar[r] & 1
	}
	$$
	
	Therefore we obtain the following.
	
	\begin{proposition}\label{proposition: reductive}
		Let $H$ be a connected reductive group and let $H^{(1)}$ denote its derived subgroup. Then
		$$
		H=H^{(1)}\rtimes \sigma(q(T_H)),
		$$
		and this splitting depends on the choice of a maximal torus of $H$.
	\end{proposition}
	
	As a consequence of Proposition \ref{proposition: reductive} we obtain the following result recently obtained by P. Popov (cf.~\cite{Popov}).
	
	\begin{corol}
		The Picard group of $H$ equals the Picard group of $H^{(1)}$.
	\end{corol}
	\begin{proof}
		By Proposition \ref{proposition: reductive} we have $H=H^{(1)}\rtimes \sigma(q(T_H))$. Thus, as a variety, $H$ is the product of $H^{(1)}$ and $\sigma(q(T_H))$. The natural projection $H\to H^{(1)}$ is a Zariski-trivial $\sigma(q(T_H))$-bundle, and ${\rm Pic}(\sigma(q(T_H)))$ is trivial; therefore it follows that the Picard group of $H$ is same as that of the Picard group of $H^{(1)}$. 
	\end{proof}
	\begin{cor}
	Assume that $H$ is a connected reductive algebraic group. Then the coordinate ring of $H$ is UFD if and only if the coordinate ring of $H^{(1)}$ is a UFD. 
	\end{cor}
	\begin{proof}
	Since the Picard group of $H$ is trivial if and only if the Picard group of $H^{(1)}$ is trivial. Rest of the proof follows as both $H$ and $H^{(1)}$ are smooth and affine.
	\end{proof}
	
	\begin{remark}
		For any connected reductive group $H$ there is a short exact sequence of Picard groups
		$$
		0\to X(Z_{H})/X(H)\to {\rm Pic}(H/Z_{H})\to {\rm Pic}(H)\to 0,
		$$
		(cf. \cite{Knop-Kraft-Luna-Vust}). In particular, when $H={\rm GL}_{n}$ we have $H/Z_H={\rm PGL}_{n}$ and ${\rm Pic}(H)=0$, hence ${\rm Pic}(H/Z_{H})=\ZZ/n\ZZ$.
	\end{remark}

	\section{Algebraic and Analytic Picard groups}\label{Sec5}
	
	Let $H$ be a connected complex algebraic group. Recall that $X(H)$ denotes the character group of $H$, i.e., the group of algebraic homomorphisms $H\to \gm$. 
	If $\chi\in X(H)$, then
	$$
	\chi_*:\pi_1(H)\to \pi_1(\gm)=\ZZ
	$$
	is a homomorphism of abelian groups. The following lemma summarizes the relation between characters and $\pi_1$.
	
	Throughout this article we use $H^i(-, \ZZ)$ to denote the singular cohomology with integral coefficients. 
	
	\begin{lemma}\label{lemma: natural}
		Let $H$ be a connected complex algebraic group. Then there is a natural homomorphism
			$$
			\Psi_{H}: X(H)\longrightarrow {\rm Hom}(\pi_{1}(H), \ZZ)=H^1(H, \ZZ),
			$$
			which is functorial in $H$. In fact, $\Psi_{H}$ is an isomorphism.
	\end{lemma}
	
	\begin{proof}
		Define
		\[
		\Psi_{H}: X(H)\longrightarrow {\rm Hom}(\pi_{1}(H), \ZZ)=H^1(H, \ZZ)
		\]
		by $\Psi_{H}(\chi)=\chi_{*}$, where $\chi_{*}$ is induced on $\pi_1$ by $\chi:H\to \gm$.

		The map $\Psi_H$ is a group homomorphism: for $\chi_1,\chi_2\in X(H)$ the usual diagram comparing $\chi_1+\chi_2$ with $\chi_1,\chi_2$ shows $\Psi_H(\chi_1+\chi_2)=\Psi_H(\chi_1)+\Psi_H(\chi_2)$.
		
			\[
		\begin{array}{c@{\hspace{2cm}}c}
			\xymatrix{
				H \ar[d]_-{\Delta} \ar[r]^-{\chi_1+\chi_2} & \gm \\
				H \times H \ar[r]_-{(\chi_1,\, \chi_2)} & \gm \times \gm \ar[u]_-{m}
			}
			&
			\xymatrix{
				\pi_1(H) \ar[d]_-{\Delta_*} \ar[r]^-{(\chi_1+\chi_2)_*} & \ZZ \\
				\pi_1(H) \times \pi_1(H) \ar[r]_-{(\chi_{1*},\, \chi_{2*})} & \ZZ \times \mathbb{Z} \ar[u]_-{m_*}
			}
		\end{array}
		\]
		
		Functoriality with respect to homomorphisms $f:H'\to H$ is immediate from naturality of $\pi_1$.
		\[ 
		\begin{array}{c@{\hspace{2cm}}c}
			\xymatrix{
				X(H) \ar[d]_-{f^{*}} \ar[r]^-{\Psi_H} & {\rm Hom}(\pi_{1}(H), \ZZ) \ar[d]_{}\\
				X(H')\ar[r]_-{\Psi_{H'}} &{\rm Hom}(\pi_{1}(H'), \ZZ)}
		\end{array}
		.\]
		
		 If $H$ is a torus of dimension $r$, then $X(H)\simeq {\rm Hom}({\rm Hom}(\gm,H), \ZZ)\simeq{\rm Hom}(\ZZ^r,\ZZ)$, and $\pi_1(H)=\ZZ^r$, so $\Psi_H$ is an isomorphism.
		
		For a reductive group $H$, Proposition \ref{proposition: reductive} yields $H=H^{(1)}\rtimes H'$ with $H'\simeq \sigma(q(T_H))$. The projection $H\to H'$ induces an isomorphism $X(H')\to X(H)$, and $\pi_1(H)\to \pi_1(H')$ is an isomorphism modulo the torsion subgroup $\pi_1(H^{(1)})$, hence $\Psi_H$ is an isomorphism.
		
		Finally, if $H$ is any connected algebraic group, let $H_u$ be its unipotent radical. The exact sequence
		$$
		1\to H_u\to H\to H/H_u\to 1
		$$
		splits, and the projection induces $X(H/H_u)\simeq X(H)$. Consider $H'=H/H_u$. Then $H'$ is reductive. Thus by the previous paragraph applies; since $\pi_1(H_u)=0$, by using the homotopy long exact sequence the map $\pi_1(H)\to\pi_1(H')$ is an isomorphism. 
	\end{proof}
	
	\begin{remark}
		One can rephrase Lemma \ref{lemma: natural} as follows.	Given $\chi\in X(H)$, let $\chi^*:H^1(\gm,\ZZ)\to H^1(H,\ZZ)$ be the induced map in cohomology. Then $\Psi_H(\chi)$ is the image of the canonical generator of $H^1(\gm,\ZZ)$ under $\chi^*$. \hfill$\Box$
	\end{remark}
	
	\subsection{Characteristic map}
	
	Let $H$ be a connected complex algebraic group and let $M$ be a path-connected topological space. Suppose $p:E\to M$ is a principal $H$-bundle with $\pi_1(E)=0$ and $\pi_2(E)=0$.
	
	Given $\chi\in X(H)$, let $L_\chi$ denote the line bundle on $M$ associated to $p$ via $\chi$. This defines a homomorphism
	\begin{equation}\label{temp1}
		X(H)\to \pic^{\rm an}(M),\qquad \chi\mapsto L_\chi.
	\end{equation}
	Let
	\begin{equation}\label{temp2}
		\pic^{\rm an}(M)\to H^2(M,\ZZ),\qquad L\mapsto c_1(L)
	\end{equation}
	be the first Chern class, and let
	\begin{equation}\label{temp3}
		H^2(M,\ZZ)\to \Hom(H_2(M),\ZZ)
	\end{equation}
	be the map from the universal coefficient theorem. The Hurewicz map $\pi_2(M)\to H_2(M)$ induces
	\begin{equation}\label{temp4}
		\Hom(H_2(M),\ZZ)\to \Hom(\pi_2(M),\ZZ).
	\end{equation}
	
	The long exact sequence of homotopy for the fibration $p:E\to M$ reads
	$$
	\cdots\to \pi_2(E)\to \pi_2(M)\to \pi_1(H)\to \pi_1(E)\to \cdots.
	$$
	Under our hypotheses $\pi_2(E)=0$ and $\pi_1(E)=0$, so the connecting homomorphism gives an isomorphism $\pi_2(M)\simeq \pi_1(H)$. This yields an isomorphism
	\begin{equation}\label{temp5}
		\Hom(\pi_2(M),\ZZ)\xlongrightarrow{\;\simeq \;} \Hom(\pi_1(H),\ZZ)=H^1(H,\ZZ).
	\end{equation}
	Composing the maps \eqref{temp1}--\eqref{temp5} we obtain
	\begin{equation}\label{temp6}
		\Psi_E:X(H)\to H^1(H,\ZZ).
	\end{equation}
	
	If $h:M'\to M$ is a continuous map of path-connected spaces and $E'=h^*E$ satisfies $\pi_1(E')=0$, $\pi_2(E')=0$, then $\Psi_{E}=\Psi_{E'}$. 
	
	Milnor constructed a contractible space $EH$ equipped with a free action of $H$. The quotient space $BH=EH\big/H$ is called the \emph{classifying space} of $H$. For any space $X$ the set of isomorphism classes of principle $H$-bundles is canonically identified with the set of homotopy classes of maps $[X, BH]$.
	
	In particular, applying the above discussion for the universal principal bundle \[EH\to BH,\] we obtain $\Psi_E=\Psi_{EH}$. We write $\Psi'_H=\Psi_{EH}$.

	\begin{prop}\label{psi-psi'}
	Let $\Psi_H$ and $\Psi'_H$ be the notation as above. Then we have	$\Psi_H=\Psi'_H$.
	\end{prop}
	
	We prove this proposition in a sequence of lemmas.
	
	\begin{lemma}\label{lemma: homotopy-pair}
		Let $p:E\to B$ and $p':E'\to B'$ be fiber bundles.  Let $f:B\to B'$, $g:E\to E'$ be maps such that the following diagram commutes

		\centerline{\xymatrix{
				E\ar[d]_p\ar[r]^{g} & E'\ar[d]^{p'}\\
				B\ar[r]_f & B'
		}}
	 Fix $b_0\in B$, $b_0'=f(b_0)$, and write $F=p^{-1}(b_0)$, $F'=p'^{-1}(b_0')$. Then the induced diagram of homotopy long exact sequences is commutative, where the rows are exact:
	 
	 \centerline{\xymatrix{
	 		\cdots \ar[r] & \pi_n(F)\ar[d]^{g_*}\ar[r] & \pi_n(E)\ar[d]^{g_*}\ar[r] & \pi_n(B)\ar[d]^{f_*}\ar[r] & \pi_{n-1}(F)\ar[d]^{g_*}\ar[r] & \cdots \\
	 		\cdots \ar[r] & \pi_n(F')\ar[r] & \pi_n(E')\ar[r] & \pi_n(B')\ar[r] & \pi_{n-1}(F')\ar[r] & \cdots \\
	 }}
	\end{lemma}
	\begin{proof}
	Follows from the proof of \cite[Chapter ~4, Theorem 4.41]{Hatcher} and the functoriality of the connecting map for the homotopy groups of a pair. 
	\end{proof}
	
	\begin{lemma}
	Let $\Psi'_H$ be the map defined as above. Then $\Psi'_H$ is functorial in $H$.
	\end{lemma}
	\begin{proof}
	Let $f:H\to H'$ be a homomorphism of algebraic groups. Let $Bf:BH\to BH'$ be the induced map of classifying spaces. Then the pullback $(Bf)^{*}EH'$ is a principal $H'$-bundle on $BH$. Note that this is equivalent to the principal $H'$-bundle obtained from $EH$ by extension of structure groups via $f$. Let $\chi'\in X(H')$ and let $\chi=\chi'\circ f$. Thus there is a commutative diagram (which is not Cartesian in general), where the typical fibers are $H$ and $H'$, respectively:
	
	\centerline{\xymatrix{
			EH\ar[d]\ar[r] & EH'\ar[d]\\
			BH\ar[r]_{Bf} & BH'
	}}
	
	Therefore by Lemma~\ref{lemma: homotopy-pair}, one has the following diagram, whose rows are parts of the homotopy long exact sequences:
	
	\centerline{\xymatrix{
			\pi_2(BH)\ar[r]\ar[d]_{\pi_2(Bf)} & \pi_1(H)\ar[d]^{f_*}\\
			\pi_2(BH')\ar[r] & \pi_1(H')
	}}
	
	This shows that (\ref{temp5}) is functorial in $H$.
	
	The maps (\ref{temp2})--(\ref{temp4}) depend entirely on the map $Bf:BH\to BH'$, hence all the relevant diagrams are commutative.
	
	If $L_{\chi'}$ is the line bundle on $BH'$ associated to $EH'$ via $\chi'$, then $(Bf)^*L_{\chi'}$ is isomorphic to $L_\chi$, the line bundle on $BH$ associated to $EH$ via the character $\chi$. This means, the map (\ref{temp1}) is functorial in $H$, in the sense that the following diagram commutes:
	
	\centerline{\xymatrix{
			X(H')\ar[d]_{f^*}\ar[r] & \pic^{\rm an}(BH')\ar[d]^{f^*}\\
			X(H)\ar[r] & \pic^{\rm an}(BH)
	}}
	
	This completes the proof that $\Psi'_H$ is functorial in $H$.
	
	\end{proof}
	
	\begin{lemma}
		Under $\Psi'_{\gm}$ the image of the identity character $(\mathrm{id}:\gm\to\gm)$ is the canonical generator of $H^1(\gm)$.
	\end{lemma}
	\begin{proof}
		Consider the principal $\gm$-bundle $p:E=\CC^2\setminus\{0\}\to \PP^1$ corresponding to $\OO(1)$. Since $E$ is homotopy equivalent to $3$-sphere $S^3$, we have $\pi_1(E)=\pi_2(E)=0$.  Note that we have an identification $\PP^1$ with $2$-sphere $S^2$. Under the composition of maps (\ref{temp1})--(\ref{temp4}), the image of the generator $(id:\gm\to \gm)\in X(\gm)$ corresponds to the homomorphism $\pi_2(\PP^1)\to \ZZ$, which takes the generator $(id:S^2\to \PP^1)\in \pi_2(\PP^1)$ to $1$. Under the connecting map $\pi_2(\PP^1)\to \pi_1(\gm)$, the generator $(id:S^2\to \PP^1)$ goes to the generator $(id:S^1\hkr \CC-\{0\})$. This completes the proof.
	\end{proof}
	
	\noindent
	\begin{proof}[Proof of \cref{psi-psi'}]
	Let $\chi\in X(H)$. For the group homomorphism $\chi:H\to \gm$, consider the commutative diagram given by the functoriality of $\Psi'$:
	$$
	\xymatrix{
		X(\gm)\ar[r]^{\Psi'_{\gm}}\ar[d]^{\chi^*} & H^1(\gm, \ZZ)\ar[d]^{\chi^*}\\
		X(H)\ar[r]_{\Psi'_H} & H^1(H,\ZZ).
	}
	$$

	By the remark following Lemma \ref{lemma: natural}, $\Psi_H(\chi)$ is the image of the canonical generator of $H^1(\gm,\ZZ)$ under $\chi^*$. The previous lemma identifies $\Psi'_{\gm}(\mathrm{id})$ with that generator, and the commutativity of the diagram yields $\Psi_H=\Psi'_H$. 
	
	\end{proof}

\begin{lemma} \label{injective-v-v'} 
	If
	\[
	0 \;\to\; V \stackrel{i}{\longrightarrow} A \stackrel{f}{\longrightarrow} B 
	\stackrel{j}{\longrightarrow} V'
	\]
	is an exact sequence of abelian groups, where $V$ and $V'$ are vector spaces,  then for every $n$, the induced map 
	\[
	\overline{f} : A_n \;\to\; B_n
	\]
	is injective.
	
\end{lemma}
\begin{proof}
	Let $a \in A$ be such that $\overline{f}(\overline{a}) = 0$.  
	Then there exists some $b \in B$ such that $f(a) = n b$.  
	But
	\[
	0 = j f(a) = n j(b),
	\]
	and since $V'$ is a vector space, we have $j(b) = 0$.  
	Hence there exists $a' \in A$ such that $b = f(a')$.  
	Thus
	\[
	f(a - n a') = 0,
	\]
	so $a - n a' = i(v)$ for some $v \in V$.  
	As $V$ is a vector space, we may write
	\[
	a = n a' + n i\!\left(\tfrac{v}{n}\right) = n a'',
	\]
	where $a'' = a' + i(v/n)$.  
	Therefore, $\overline{a} = 0$ in $A_n$.
\end{proof}

Let $G$ be a semisimple, simply connected and let $H$ be a closed connected subgroup of $G$. For $\chi\in X(H)$, denote by $L_\chi\in\pic(G/H)$ the associated line bundle. 

Consider the maps
\[
c_1\colon \pic^{\rm an}(G/H)\to H^2(G/H,\ZZ),
\quad
H^2(G/H,\ZZ)\to \Hom(H_2(G/H),\ZZ),
\]
the Hurewicz map $\Hom(H_2(G/H),\ZZ)\to \Hom(\pi_2(G/H),\ZZ)$, and the identification \[\Hom(\pi_2(G/H),\ZZ)=\Hom(\pi_1(H),\ZZ)=H^1(H,\ZZ).\] Let
\[
\alpha\colon \pic^{\mathrm{an}}(G/H)\to H^1(H,\ZZ)
\]
be the composition of these maps.

\begin{prop}\label{commute-1}
	With the notation as above, the following diagram commutes
	\[
	\xymatrix{
		X(H)\ar[r]\ar[d]^{\Psi_H} & \pic^{}(G/H)\ar[d] \\
		H^1(H,\ZZ) & \pic^{\mathrm{an}}(G/H)\ar[l]_\alpha
	}
	\]
	where the upper horizontal map sends $\chi$ to $L_\chi$.
\end{prop}
\begin{proof}
	The universal principal bundle $EH\to BH$ identifies the construction used to define $\Psi'_H$ with the composition
	\[
	X(H)\to \pic^{}(G/H)\to \pic^{\mathrm{an}}(G/H)\xrightarrow{\alpha} H^1(H,\ZZ),
	\]
	and Proposition \ref{psi-psi'} shows this equals $\Psi_H$, hence the diagram commutes.
\end{proof}

\begin{lemma}\label{lemma: alg pic}
	Let $G$ be a connected linear algebraic group such that ${\rm Pic}(G)=0$. Let $H$ be a closed subgroup of $G$. Then ${\rm Pic^{}}(G/H)=X(H)$. 
\end{lemma}
\begin{proof}
	Indeed, any line bundle on $G/H$ is an $H$-equivariant line bundle on $G$. Since ${\rm Pic}(G)=0$, this means a line bundle on $G/H$ is given by a $H$-linearization of the trivial line bundle on $G$. By Rosenlicht's result, $\OO(G)^{\times}$ is just $\CC^{\times}$, so a $H$-linearization of the trivial line bundle on $G$ is just a character of $H$. 
\end{proof}

\begin{prop} \label{brauer-3} Let $G$ be a connected, simply connected, semisimple algebraic group and let $H$ be a closed connected subgroup of $G$. Let $M= G/H$. Then we have the following.
	\begin{enumerate}
		\item  $\pi_1(M) = 0$.
		\item For every positive integer $n$,  the natural map $$\pic^{}(M)_n \;\to\; \pic^{\rm an}(M)_n$$
		is an isomorphism, where $\pic(M)_n$ (respectively, $\pic^{\rm an}(M)_n$) denotes cokernel for the multiplication by $n$ maps.
	\end{enumerate}
	
\end{prop}
 \begin{proof}
	Form the fibration $G \to M=G/H$, we have the homotopy long exact sequence 
	
	\begin{equation}\label{homotopy-1} 
		\pi_2(G)\to\; \pi_2(M)\to\; \pi_1(H) \to\; \pi_1(G) \to\; \pi_1(M) \to\; \pi_0(H).
	\end{equation}
	
	Since $G$ is a Lie group, $\pi_2(G) = 0$. Moreover, since $G$ is simply connected, $\pi_1(G) = 0$; and as $H$ is connected, $\pi_0(H) = 0$.  Hence 
	\[
	\pi_2(M) = \pi_1(H), 
	\quad 
	\pi_1(M) = 0.
	\]
	Thus $H_1(M,\ZZ) = 0$, so $H^1(M,\ZZ) = 0$.
	
	\medskip
	
	By using the exponential sequence 
	\[
	0 \to \ZZ \to \OO_{M}^{an} \stackrel{e^{2\pi i z}}{\longrightarrow} \OO_{M}^{an, \times} \to 1,
	\]
	we obtain the long exact sequence 
	\begin{equation}\label{pic-an-1} 
		H^1(M,\ZZ) \to H^1(M,\OO_{M}^ {an}) \to \pic^{\rm an}(M) \stackrel{c_1}{\longrightarrow} H^2(M,\ZZ) \to H^2(M,\OO_{M}^ {an}).
	\end{equation}
	
	Since $H^1(M,\ZZ) = 0$, by the universal coefficient theorem map $H^2(M,\ZZ) \to \Hom(H_2(M),\ZZ)$ and the Hurewicz map $\Hom(H_2(M),\ZZ) \to \Hom(\pi_2(M),\ZZ)$ are isomorphisms.  
	Therefore 
	\[
	\Hom(\pi_2(M),\ZZ) = \Hom(\pi_1(H),\ZZ) = H^1(H,\ZZ).
	\]
	Therefore the Chern class map $c_1$ in \eqref{pic-an-1} can be identified with 
	\[
	\alpha : \pic^{\rm an}(M) \;\to\; H^1(H,\ZZ),
	\]
	where $\alpha$ is as in Proposition~\ref{commute-1}.
	
	\medskip
	Thus we have an exact sequence 
	\begin{equation}\label{exact-v-v'}  
		0 \to V \to  \pic^{\rm an}(M) \stackrel{\alpha}{\longrightarrow} H^1(H,\ZZ) \to V',
	\end{equation}
	where $V = H^1(M,\OO_{M}^ {an})$ and $V' = H^2(M,\OO_{M}^ {an})$ are complex vector spaces.
	
	\medskip
	
	The commutative diagram in Proposition~\ref{commute-1} gives rise to the following commutative diagram for every $n$:
	\[
	\centerline{
		\xymatrix{
			X(H)_n \ar[r] \ar[d]^{\Psi_H} 
			& \pic^{}(M)_n \ar[d] \\
			H^1(H,\ZZ)_n 
			& \pic^{\rm an}(M)_n \ar[l]_{\overline{\alpha}}
	}}
	\]
	
	The left vertical map is an isomorphism by Lemma~\ref{lemma: natural}, which shows that $\overline{\alpha}$ must be surjective.  
	The upper horizontal map is an isomorphism because $X(H) \to \pic(M)$ is an isomorphism, as $G$ is simply connected.  
	By the exact sequence \eqref{exact-v-v'} and Lemma~\ref{injective-v-v'}, it follows that $\overline{\alpha}$ is injective.  
	Therefore $\overline{\alpha}$ is a bijection, hence the right vertical map is also a bijection.  
	This proves the lemma.
	
\end{proof}

\subsection{When $H\subset G$ is closed connected reductive subgroup}
Let $G$ be a connected, simply connected, semisimple group and let $H\subset G$ be a closed connected reductive subgroup. 

In this subsection, we show that the natural map
\[
\pic(G/H)\to \pic^{\mathrm{an}}(G/H),\qquad
\]
is an isomorphism.

\begin{theorem}\label{picard: alg-ana}
	Let $G$ be a connected, simply connected, semieimple complex algebraic group and $H\subset G$ be a closed connected reductive subgroup. Then the natural map
	\[
	\pic(G/H)\longrightarrow \pic^{\mathrm{an}}(G/H)
	\]
	is an isomorphism.
\end{theorem}
\begin{proof}
	By Proposition \ref{commute-1} we have the following commutative diagram.
	\[
	\xymatrix{
		X(H)\ar[r]^{\simeq \hspace{1em}}\ar[d]_{\Psi_H}^{\simeq} & \mathrm{Pic}^{}(G/H)\ar[d] \\
		\Hom(\pi_1(H),\ZZ) & \mathrm{Pic}^{\mathrm{an}}(G/H)\ar[l]^{\hspace{1em} \qquad\simeq\qquad}
	}
	\]
where the horizontal maps are isomorphisms by Lemmas \ref{lemma: alg pic} and \ref{brauer-3}, and $\Psi_H$ is an isomorphism by Lemma \ref{lemma: natural}. Hence the right vertical map is an isomorphism.
\end{proof}

\begin{remark}\label{remark: picard}
	If $H$ is not reductive, then $\mathrm{Pic}^{\mathrm{an}}(G/H)\to \Hom(\pi_1(H),\ZZ)$ need not be an isomorphism. Consequently $\mathrm{Pic}(G/H)\to \mathrm{Pic}^{\mathrm{an}}(G/H)$ may fail to be an isomorphism. For example, take $G=\mathrm{SL}_2$ and let $H$ be the subgroup of invertible strictly upper triangular matrices. Then $G/H\simeq \CC^2\setminus\{0\}$. Here $X(H)=0$, so $\mathrm{Pic}^{}(G/H)=0$, while $\mathrm{Pic}^{\mathrm{an}}(G/H)$ is infinite (cf. \cite[Ex.~4, p.~49]{Griffiths-Harris}). Set $M=G/H$. The long exact sequence arising from the exponential sequence is
	\begin{multline*}
		\cdots \to H^{1}\left(M, \mathbb{Z}\right) 
		\to H^{1}\left(M, \OO_{M}^{an}\right) 
		\to \mathrm{Pic}^{\rm an}\left(M\right) 
		\to H^{2}\left(M, \mathbb{Z}\right) 
		\to \cdots
	\end{multline*}
	
	\noindent Since $M$ is homotopy equivalent to $S^{3}$, both 
	$H^{1}\!\left(M, \mathbb{Z}\right)$ and 
	$H^{2}\!\left(M, \mathbb{Z}\right)$ vanish. 
	Consequently,
	\[
	\mathrm{Pic}^{\rm an}\!\left(M^{}\right) 
	\;\simeq\; H^{1}\!\left(M, \OO_{M}^{an}\right),
	\]
	which is infinite-dimensional (cf. \cite[Ex.~4, p.~49]{Griffiths-Harris}).
	
\end{remark}

\section{Extension of algebraic groups and extension of abelian groups}\label{section6.1}
In this section, we recall from Kumar--Neeb (cf. \cite{Kumar-Neeb}) some properties of the extension groups of connected complex algebraic groups by tori.

Let $H$ be a connected complex algebraic group. Let $H=H_{u}\rtimes H_{\rm red}$ be a Levi decomposition of $H$, where $H_u$ is the unipotent radical of $H$ and $H_{\rm red}$ is a connected reductive subgroup of $H$. Since $H_{u}$ is simply connected, (i.e., $\pi_{1}(H_u)$ is trivial) $\pi_{1}(H)\simeq \pi_{1}(H_{\rm red})$.

Again, $H_{\rm red}=Z\cdot H^{(1)}_{\rm red}$, where $Z$ is the connected component of the center of $H_{\rm red}$, and $H^{(1)}_{\rm red}$ is the commutator subgroup of $H_{\rm red}$. Then there is an algebraic universal cover \[H^{(1)}_{\rm red, sc}\xrightarrow{f} H^{(1)}_{\rm red}.\] Thus we have a short exact sequence
\begin{equation}\label{pi1}
	1\to \pi_1(H^{(1)}_{\rm red})\to H^{(1)}_{\rm red, sc}\xrightarrow{f} H^{(1)}_{\rm red}\to 1,
\end{equation}
where  $H^{(1)}_{\rm red, sc}$ is a connected, simply connected, semisimple algebraic group and fundamental group $\pi_1(H^{(1)}_{\rm red})$ can be identified with a finite central subgroup of $H^{(1)}_{\rm red, sc}$. Consider the covering group homomorphism \[\widetilde{H}_{\rm red}:=Z\times H^{(1)}_{\rm red, sc}\to Z\cdot H^{(1)}_{\rm red}=H_{\rm red}\] sending $(z,h)\mapsto zf(h)$. Let $\Pi_H$ be the kernel of this map. So, we have a short exact sequence \[1\to \Pi_H\to \widetilde{H}_{\rm red}\to H_{\rm red}\to 1.\] 
Note that $\Pi_H=\{(a,b)\in Z\times H^{(1)}_{\mathrm{red,sc}} \;:\; a^{-1}=f(b)\}
\cong \{b\in H^{(1)}_{\mathrm{red,sc}} :\, f(b)\in Z\cap H^{(1)}_{\mathrm{red}}\}=f\inv(Z\cap H^{(1)}_{\mathrm{red}}),$ which is an extension of $Z\cap H^{(1)}_{\mathrm{red}}$ by  $\pi_1(H^{(1)}_{\rm red})$.  
Since $H^{(1)}_{\mathrm{red}}$ is semisimple and $Z$ is central in $H_{\mathrm{red}}$, 
the intersection $Z\cap H^{(1)}_{\mathrm{red}}$ is finite. 
Moreover, the fundamental group 
$\pi_1(H^{(1)}_{\mathrm{red}})$ of the semisimple group 
$H^{(1)}_{\mathrm{red}}$ is finite. Therefore,
$\Pi_H$
is also finite. 
Now we have a covering group homomorphism $q_H:\widetilde{H}:=H_u\rtimes \widetilde{H}_{\rm red}\to H$ with $\Pi_H$ as its kernel. Thus we have a short exact sequence 
\begin{equation}\label{pi2}
	1\to \Pi_H\to \widetilde{H} \to H\to 1,
\end{equation}
where $\widetilde{H}$ is connected and $\Pi_{H}$ is a finite, hence central subgroup of $\widetilde{H}$.

Let $A$ be a torus. Then we have the natural restriction map 
\[ {\rm res} \colon \mathrm{Hom}(\widetilde{H}, A)\longrightarrow \mathrm{Hom}(\Pi_H,A)\] which is in fact  homomorphism of abelian groups.

Given an algebraic group homomorphism $\gamma:\Pi_H\to A$, the pushout of \eqref{pi2} by $\gamma$ gives rise to a central extension
\[
\begin{array}{c@{\hspace{2cm}}c}
	\xymatrix{
		\hspace{1.6cm}1\ar[r] & \Pi_H  \ar[d]_-{{\gamma }} \ar[r]_-{}\ar[r]& \widetilde{H} \ar[r]\ar[d]^{} & H\ar[r]\ar@{=}[d]& 1\\
	\Phi(\gamma):=	1\ar[r]&A\ar[r]& (\widetilde{H}\times A) \big/ \Pi_{H}\ar[r]_-{} &H\ar[r]& 1}
\end{array}
\]
Therefore we have a natural homomorphism of abelian groups 
\[
\Phi \colon \mathrm{Hom}(\Pi_H,A)\longrightarrow \mathrm{E}_{\mathrm{al}}(H,A),
\]
where $\mathrm{E}_{\mathrm{al}}(H,A)$ is an abelian group with respect to the Brauer sum of central extensions. Then we have the following result.
\begin{theorem}\label{Thm: kumar-neeb} Let $A$ be a torus and with notation as above. Then following diagram below commutes 
\[
\xymatrix{
	\mathrm{Hom}(\Pi_H,A) 
	\ar[rr]^{\Phi} 
	\ar[d]_{\mathrm{rest.\,map}}\ar[drr]^{\mathrm{quot. map}} 
	&& \mathrm{E}_{\mathrm{al}}(H,A) \\
	\mathrm{Hom}(\pi_1(H^{(1)}_{\rm red}),A)
	&&  
	\dfrac{\mathrm{Hom}(\Pi_H,A)}{\rm res(\mathrm{Hom}(\widetilde{H},A))} \ar[u]_{\Phi^{'}} \ar[ll]^{\mathrm{isom.}} &
}
\] and 
\begin{enumerate}
	\item  The natural map $\Phi$ is  functorial with respect to both $H$ and $A$.
	\item The natural restriction map \[\mathrm{Hom}(\Pi_H,A)\to \mathrm{Hom}(\pi_1(H^{(1)}_{\rm red}),A) \] factors through an isomorphism \[\dfrac{\mathrm{Hom}(\Pi_H,A)}{\rm res(\mathrm{Hom}(\widetilde{H},A))} \to 	\mathrm{Hom}(\pi_1(H^{(1)}_{\rm red}),A)\].
	\item The induced map $\Phi^{'}$ is an isomorphism.
	\item There is an isomorphism from ${\rm Hom}(\pi_1(H^{(1)}_{\rm red}), A)$ onto $\mathrm{E}_{\mathrm{al}}(H,A)$, which we also denote by $\Phi^{'}$
\end{enumerate}	
\end{theorem}
\begin{proof}
See \cite[Proposition 1.4]{Kumar-Neeb}.
\end{proof}

Note that $\pi_{1}(H^{(1)})=\pi_{1}(H^{(1)}_{\rm red})$ is a finite group. Thus, \cref{Thm: kumar-neeb}(iv) gives rise
to an isomorphism (cf. \cite[Theorem 1.8(b)]{Kumar-Neeb})
\[\Phi^{'}: \mathrm{Hom}(\pi_1(H^{(1)}),A)\to 	\mathrm{E}_{\rm al}(H, A).\]
Consider the short exact sequence \[1\to H^{(1)}\to H\to H/H^{(1)}\to 1.\]  The long exact sequence of homotopy sequences gives a short exact sequence $$0\to \pi_1(H^{(1)})\to \pi_1(H)\to \pi_1(H/H^{(1)})\to 0.$$ As $H/H^{(1)}$ is a torus, its fundamental group is free abelian of finite rank. Thus, ${\rm Ext}^1(\pi_1(H),\ZZ)\simeq{\rm  Ext}^1(\pi_1(H^{(1)}),\ZZ)$.

Let $\Gamma$ be a finite abelian group and let $\chi\in X(\Gamma):={\rm Hom}(\Gamma, \gm)$. The pullback of the exponential sequence
\begin{equation}\label{exp}
	0\to \ZZ\to \CC\to \CC^\times \to 1
\end{equation}
by $\chi$ defines an abelian extension of $\Gamma$ by $\ZZ$, hence a class in $\mathrm{Ext}^1(\Gamma,\ZZ)$. This process defines a group homomorphism
\[
e_\Gamma\colon X(\Gamma)\to \mathrm{Ext}^1(\Gamma,\ZZ).
\]

\begin{lemma}
	The map $e_\Gamma$ is an isomorphism.
\end{lemma}
\begin{proof}
	There is an equivalent, homological, description of $e_\Gamma(\chi)$. Applying $\Hom(\Gamma,-)$ to \eqref{exp} yields the long exact sequence
	\[
	\cdots\to \Hom(\Gamma,\CC)\to \Hom(\Gamma,\gm)\xrightarrow{\delta} \mathrm{Ext}^1(\Gamma,\ZZ)\to \mathrm{Ext}^1(\Gamma,\CC)\to\cdots.
	\]
	By definition $e_\Gamma(\chi)=\delta(\chi)$. Since $\Gamma$ is finite, $\Hom(\Gamma,\CC)=0$ and $\mathrm{Ext}^1(\Gamma,\CC)=0$, so the connecting map $\delta$ is an isomorphism.
\end{proof}

Take $\Gamma=\pi_1(H^{(1)})$. Then the above discussion implies, the composition of the two maps $${\rm Hom}(\pi_1(H^{(1)}), \gm)\stackrel{e_{\Gamma}}\to {\rm Ext}^{1}(\pi_1(H^{(1)}),\ZZ)\xrightarrow{\simeq}{\rm Ext}^{1}(\pi_1(H),\ZZ)$$ is an isomorphism.  
Thus, from the above discussion we have the following lemma.
\begin{lemma}\label{cor:6.7}
	Let $H$ be a connected algebraic group. Then there is a natural isomorphism
	\[
\Phi_{H}: {\rm Ext}^{1}(\pi_1(H),\ZZ) \xrightarrow{\simeq^{-1}} {\rm Ext}^{1}(\pi_1(H^{(1)}),\ZZ) \stackrel{e_{\Gamma}^{-1}}\to {\rm Hom}(\pi_1(H^{(1)}), \gm) \stackrel{{\Phi^{'}}}\to \mathrm{E}_{\mathrm{al}}(H,\gm)
	\]
which is functorial with respect to $H$. 
\end{lemma}

\subsection{Algebraic and Analytic Brauer groups of a Homogeneous Space}

Let $X$ be a smooth quasi projective complex algebraic variety. Note that since $X$ is smooth, we have $\mathrm{Br}'(X)=H^{2}_{\topet}(X,\gm)$ (cf. \cite[IV.2.6]{Milne}). Since $X$ is a quasi projective complex algebraic variety, by a result of Gabber the natural inclusion from ${\rm Br}(X)$ into ${\rm Br}'(X)$ is an isomorphism (cf. \cite[Theorem 1.1]{Jong2005ARO}).

 \begin{lemma}\label{brauer-1} Let $X$ be a smooth quasi projective complex algebraic variety. Then we have the following
 	\begin{enumerate}
 		\item  For every $n$, $$\pic^{}(X)_n\to \pic^{\rm an}(X)_n$$ is injective, where $\pic(X)_n$ (respectively, $\pic^{\rm an}(X)_n$) denotes cokernel for the multiplication by $n$ maps.
 		
 		\item  Moreover, if for every $n$, $\pic^{}(X)_n\to \pic^{\rm an}(X)_n$ is surjective, then $${\rm Br}'(X)\to {\rm Br}'^{\rm an}(X)$$ is an isomorphism.
 	\end{enumerate} 
	\end{lemma}
\begin{proof} 
For a fixed integer $n$, the Kummer exact sequence gives rise to a commutative diagram in which each row is exact:
\[
\centerline{
	\xymatrix{
		0 \ar[r] & \pic(X)_n \ar[r] \ar[d] 
		& H^2_{\topet}(X,\mu_n) \ar[d]^{\simeq} \ar[r] 
		& {}_n H^2_{\topet}(X,\gm) \ar[d] \ar[r] & 0 \\
		0 \ar[r] & \pic^{\rm an}(X)_n \ar[r] 
		& H^2(X,\mu_n) \ar[r] 
		& {}_n H^2(X,\OO_{X}^{an, \times}) \ar[r] & 0
}}
\]
where $_{n}H^2\topet(X, \gm)$ and ${\rm Pic}(X)_{n}$ are the kernel and cokernel for multiplication-by-$n$ map.

\noindent By \cite[Theorem 3.12]{Milne} the middle vertical map is an isomorphism, so (i) follows immediately. For (ii), assume the left vertical map is surjective.  The contention follows from the five lemma and the fact that both the Brauer groups involved are torsion groups.
\end{proof}

\begin{theorem} \label{brauer-4} Let $G$ be a connected, simply connected,  semisimple algebraic group and let $H$ be a closed connected subgroup of $G$. Let $M = G/H$. Then the natural map
	\[
	{\rm Br}'(M) \;\longrightarrow\; {\rm Br}'^{\rm an}(M)
	\]
	is an isomorphism, 	where ${\rm Br}'(M)$ (respectively, ${\rm Br}'^{\rm an}(M)$) denotes the cohomological Brauer group of $M$.
\end{theorem}

\begin{proof}
	Follows from Lemma \ref{brauer-3} and Lemma \ref{brauer-1}.
\end{proof}

\section{Main Theorem:  Brauer group of homogeneous space}\label{Sec7}
In this section, we compute the analytic and algebraic Brauer groups of the homogeneous space for the action of a connected, simply connected, semisimple complex algebraic group with stabilizer connected algebraic group.

Let $X$ be a smooth quasi-projective variety. Then the exponential sequence
$$
0\to \ZZ \to \OO_{X}^{an} \to \OO^{an, \times}_{X}\to 1
$$
gives rise to a long exact sequence 
\[\pic^{\rm an}(X)\stackrel{c_1}\to H^2(X,\ZZ)\to H^2(X,\OO_{X}^{an})\to H^2(X,\OO_{X}^{an, \times})\xrightarrow{\delta} H^3(X,\ZZ)\to H^3(X,\OO_{X}^{an}),\] where $c_1$ denotes the first Chern class map.
Recall that ${\rm Br'}^{\rm an}(X)$ is defined as the torsion subgroup ${\rm Tors}(H^2(X,\OO_{X}^{an, \times}))$ of $H^2(X,\OO_{X}^{an, \times})$.
\begin{lem}\label{lemma 7.1}
Let $X$ be a smooth quasi-projective variety. If the first Chern class map $c_1$ is surjective, then the induced map \[\overline{\delta}: {\rm Br'}^{\rm an}(X)\to {\rm Tors}(H^3(X,\ZZ))\] is an isomorphism.
\end{lem}
\begin{proof}  
First we see that the induced map $\overline{\delta}$ is injective. 
Since $c_1$ is surjective, $H^2(X, \ZZ)\to H^2(X, \OO_{X}^{an})$ is zero map.  Since $H^2(X, \OO_{X}^{an})$ is a vector space, it is torsion free. Moreover, since ${\rm Br'}^{\rm an}(X)$ is a torsion group, it follows that $\overline{\delta}$ is injective.

Next we show that $\overline{\delta}$ is surjective.
Note that ${\rm Im}(\delta)=\ker\big(H^3(X, \ZZ)\to H^3(X, \OO_{X}^{an})\big )$. Since $H^3(X, \OO_{X}^{an})$ is a vector space, it is torsion free. Hence, ${\rm Tors}(H^3(X,\ZZ))\subseteq {\rm Im}(\delta)$.
Assume that $0\neq \beta \in {\rm Tors}(H^3(X,\ZZ))$. Then there exists $0\neq \alpha\in H^2(X,\OO_{X}^{an, \times})$ such that $\delta (\alpha)=\beta$. Since $\beta$ is torsion, there exists a positive integer $m$ such that $m\beta=0$. This implies that $\delta(m\alpha)=0$. Hence, $m\alpha\in {\rm Im}(H^2(X,\OO_{X}^{an})\to H^2(X,\OO_{X}^{an, \times}))$. Since $H^2(X,\OO_{X}^{an})$ is a vector space, if $0\neq m\alpha\in {\rm Im}(H^2(X,\OO_{X}^{an})\to H^2(X,\OO_{X}^{an, \times})),$ then $\alpha\in {\rm Im}(H^2(X,\OO_{X}^{an})\to H^2(X,\OO_{X}^{an, \times}))$. This shows that $\beta= 0$, which is a contradiction. Thus we have $m\alpha=0$. Hence, $\alpha \in {\rm Br'}^{\rm an}(X)$. Therefore, $\overline{\delta}$ is surjective.

\end{proof}

\begin{prop}\label{Analytical-Brauer}
Let $G$ be a connected, simply connected, semisimple complex algebraic group, and $H$ be a closed connected subgroup of $G$. Consider the homogeneous space $X=G/H$. Then the analytic cohomological Brauer group ${\rm Br'}^{\rm an}(X)$ is isomorphic ${\rm Ext}^{1}(\pi_{1}(H), \ZZ)$. In particular, ${\rm Br'}^{\rm an}(X)$  is finite.
\end{prop}
\begin{proof}
By \cref{brauer-3}, it follows that $\pi_1(X)=0$. Since $\pi_{1}(X)=0$, by Hurewicz theorem $H_2(X, \ZZ)=\pi_2(X)$. Further, by using homotopy longe exact sequence for the fibration $G\to X$, we obtain $\pi_2(X)=\pi_1(H)$. Thus, $H_2(X, \ZZ)=\pi_{1}(H)$. Also, by using the universal coefficient theorem, $H^2(X,\ZZ)={\rm Hom}(H_2(X),\ZZ)={\rm Hom}(\pi_1(H),\ZZ)$. 

Next we show that $c_1$ is surjective. Notice that by \cref{commute-1}, $c_1$ fits in the below commutative diagram. 
\[
\xymatrix{
	X(H)\ar[r]^{\simeq}\ar[d]_{\Psi_H}^{\simeq} & \pic(X)\ar[d] \\
	\Hom(\pi_1(H),\ZZ) & \pic^{\rm an}(X)\ar[l]^{\qquad c_1}
}
\]
By \cref{lemma: natural} the left vertical map \[\Psi_H: X(H)\to {\rm Hom}(\pi_1(H),\ZZ)\] is an isomorphism. Therefore, $c_1$ is surjective. Therefore, by Lemma~\ref{lemma 7.1}, the induced map \[{\rm Br'}^{\rm an}(X)\to {\rm Tors}(H^3(X,\ZZ))\] is an isomorphism. By using the universal coefficient theorem, \[{\rm Tors}(H^3(X,\ZZ))\simeq {\rm Ext}^{1}(H_2(X),\ZZ)\simeq {\rm Ext}^{1}(\pi_1(H),\ZZ).\] Moreover, we have \[{\rm Ext}^{1}(\pi_1(H),\ZZ)\simeq {\rm Ext}^{1}({\rm Tors}(\pi_1(H)),\ZZ).\] Since ${\rm Tors}(\pi_{1}(H))$ is finite, it follows that ${\rm Ext}^{1}(\pi_1(H),\ZZ)$ is finite.

\end{proof}

\medskip

We now proceed to the proof of our main theorem (cf. \cref{brauergp}).
\begin{proof}[Proof of \cref{brauergp} ]\label{proof of theorem 3.2: analytic-brauer}
By \cref{Analytical-Brauer}, we have an isomorphism
\[
{\rm Br'}^{\rm an}(G/H)\;\simeq\; {\rm Ext}^{1}(\pi_{1}(H),\ZZ).
\]
Moreover, by \cref{brauer-4}, it follows that
\[
{\rm Br'}^{\rm an}(G/H)\;\simeq\;{\rm Br'}(G/H)\;\simeq\; {\rm Br}(G/H),
\]
where the last isomorphism follows from a result of Gabber (cf.\cite[Theorem 1.1]{Jong2005ARO}) as $G/H$ is a quasi projective variety.

By \cite[Theorem 5.1]{Haboush}, the natural map
\[
{\rm E}_{\rm al}(H,\gm)\longrightarrow {\rm Br}(G/H)
\]
is injective. On the other hand, by \cref{cor:6.7}, the map
\[
\Phi_{H}:\; {\rm Ext}^{1}(\pi_{1}(H),\ZZ)  \longrightarrow\; {\rm E}_{\rm al}(H,\gm)\;
\]
is an isomorphism. Since both \({\rm E}_{\rm al}(H,\gm)\) and \({\rm Br}(G/H)\) are finite of the same cardinality, the injective map
\[
{\rm E}_{\rm al}(H,\gm)\longrightarrow {\rm Br}(G/H)
\]
is an isomorphism.
\end{proof}

\begin{theorem}\label{theorem5.5}
	Let $G$ be a connected, simply connected, semisimple complex algebraic group, and $H$ be a closed connected subgroup of $G$. Consider the homogeneous space $X=G/H$. Then there is natural isomorphism of 
	${\rm Ext}^{1}(\pi_{1}(H), \ZZ) \to {\rm Br}(X)$.
\end{theorem}
\begin{proof}
	Follows from ~\cref{brauergp} and ~\cref{cor:6.7}.
\end{proof}

\begin{theorem}\label{cor: baruer}
The natural inclusion map ${\rm Br}^{\rm  an}(G/H)\hookrightarrow {\rm Br'}^{\rm an}(G/H)$ is an isomorphism. In particular, the the natural map ${\rm Br}(G/H)\longrightarrow {\rm Br}^{\rm an}(G/H)$ is an isomorphism.
\end{theorem}
\begin{proof}
By \cref{brauer-4}, ${\rm Br'}(G/H)\to {\rm Br'}^{\rm an}(G/H)$ is an isomorphism. On the other hand, Since $G/H$ is a quasi projective variety, by Gabber's result it follows that the natural inclusion ${\rm Br}(G/H)\to {\rm Br'}(G/H)$ is an isomorphism  (cf. \cite[Theorem 1.1]{Jong2005ARO}).
Consider the following commutative diagram: 
\[
\xymatrix{
	{\rm Br}(G/H)\ar[r]^{\simeq}\ar[d]_{}^{} & \mathrm{\rm Br'}(G/H)\ar[d]^{\simeq} \\
	{\rm Br}^{\rm an}(G/H)\ar@{^(->}[r] & \mathrm{Br'}^{\rm an}(G/H).
}
\]  
Therefore, the bottom horizontal map is an isomorphism. The last part follows immediately.
\end{proof}

\begin{remark}
If $X$ is a smooth projective rational variety, then $H^3(X,\ZZ)$ is torsion-free, in particular, ${\rm Br}(X)$ is trivial (cf. \cite[Proposition 1.]{Artin-Mumford}). However, Example~\ref{example: brauer group non-zero} demonstrates that this need not hold for smooth non-projective rational varieties, i.e., the Brauer group is non-trivial and $H^3(X,\ZZ)\neq 0$ (cf. \cref{Analytical-Brauer} and  \cref{cor: baruer}).
\end{remark}

\bibliographystyle{amsalpha}
\bibliography{references}
	
\end{document}